\documentclass[11pt]{article}
\usepackage{amsmath, amssymb, amsthm, amsfonts}
\usepackage{indentfirst}
\usepackage[mathscr]{euscript}
\usepackage{amscd}
\usepackage{xcolor}
\usepackage[pdftex]{graphicx}

\numberwithin{equation}{section}
\numberwithin{figure}{section}

\theoremstyle{plain}
\newtheorem{theorem}{Theorem}\numberwithin{theorem}{section}
\newtheorem{lemma}{Lemma}\numberwithin{lemma}{section}
\numberwithin{proposition}{section}
\newtheorem{corollary}{Corollary}\numberwithin{corollary}{section}
\theoremstyle{definition}
\numberwithin{definition}{section}
\theoremstyle{remark}
\newtheorem{remark}{Remark}\numberwithin{remark}{section}

\newcommand{\R}{\mathbb{R}}

\title{Positivity of oscillatory integrals\\and Hankel transforms}

\author{Yong-Kum Cho\footnote{corresponding author. ykcho@cau.ac.kr. Department of Mathematics, College of Natural Sciences, Chung-Ang University,
84 Heukseok-Ro, Dongjak-Gu, Seoul 06974, Korea.}
\and  Seok-Young Chung\footnote{sychung@knights.ucf.edu. Department of Mathematics, University of Central Florida,
4393 Andromeda Loop N., Orlando, FL 32816, USA.}\\
\and Young Woong Park\footnote{ywpark1839@gmail.com. Department of Mathematics,  Chung-Ang University.}}

\date{}

\begin{document}

\maketitle

{\bf Abstract.}  In consideration of the integral transform
whose kernel arises as an oscillatory solution of certain second-order linear differential equation,
its positivity is investigated on the basis of Sturm's theory.
As applications, positivity criteria are obtained for Hankel transforms as well as trigonometric
integrals defined on the positive real line.

\medskip


{oscillatory; positivity; Sturm's oscillation theory.}

\medskip


\section{Introduction}
In 1823, C. Sturm \cite{S} founded an oscillation theory for the
second-order linear differential equation expressible in the normal form
\begin{equation}\label{ST1}
u'' + \phi(t) u =0,\quad t>0.
\end{equation}

In this paper we consider the integral transform of the type
\begin{equation}\label{OT1}
(U f)(x) = \int_0^\infty f(t) u(xt)\,dt \qquad(x>0),
\end{equation}
where $f(t)$ is a nonnegative function for $\,t>0\,$ such that the integral exists and
the kernel $u(t)$ is an oscillatory solution of the differential equation \eqref{ST1}, and aim to investigate if the inequality
$\,(Uf)(x)>0\,$ holds true for all $\,x>0\,$ on the basis of Sturm's oscillation theory.

A great deal of functions encountered in mathematical physics, which are oscillatory in nature, arise
as solutions of \eqref{ST1} or images of \eqref{OT1}. In the special case $\,\phi(t)=1,\,t>0,$ for instance,
\eqref{OT1} gives rise to Fourier sine or cosine transforms on the positive real line. Concerning
positivity, a classical theorem due to P\'{o}lya \cite{P} states that if $f(t)$ is a continuous, positive and decreasing function
such that $\,\int_0^\infty f(t) dt\,$ exists, then
\begin{align}\label{Po1}
\left\{\begin{aligned}
&{\,{\rm (i)}\,\quad \int_0^\infty f(t)\sin (xt)\,dt >0,}\\
&{{\rm (ii)}\quad \int_0^\infty f(t)\cos (xt)\,dt >0,}\end{aligned}\right.
\end{align}
for all $\,x>0,$ where (ii) holds true under the additional assumption that $f(t)$ is convex for $\,t>0\,$
 (see Tuck \cite{T} for related issues).

Another special case is the Hankel transform defined by
\begin{equation}\label{H1}
(H_\nu f)(x) = \int_0^\infty f(t)J_\nu(xt)\sqrt{xt}\, dt\qquad(x>0),
\end{equation}
where $J_\nu(z)$ stands for the Bessel function of the first kind and
order $\nu$ is assumed to be real with $\,\nu>-1.$
Owing to identities \cite[3.4]{Wa},
\begin{align*}
J_{-1/2}(z) =\sqrt{\frac{2}{\pi z}\,}\cos z, \,\, J_{1/2}(z) = \sqrt{\frac{2}{\pi z}\,}\sin z,
\end{align*}
it reduces to Fourier transforms of type \eqref{Po1} when $\,\nu=\pm 1/2.$

As it will become clearer in the sequel, we deduce positivity of $(Uf)(x)$ by decomposing it into sums of integrals over intervals
between consecutive zeros of $u(t)$ and by comparing arches of graph of integrands. An advantage of this approach is that sufficient
conditions on $f(t)$ for positivity are mild and analogous to aforementioned P\'olya's conditions. Our method,
however, requires that $\,u(0+) =0\,$ and zeros of $u(t)$ must be enumerable in ascending order of magnitude, which
could be a drawback in practice.

Concerning positivity of Hankel transforms, we remark that Zhang \cite{Z} has recently obtained a set of positivity criteria applicable to
certain class of completely monotone functions. Unlike our method, his approaches are based on Laplace transforms of Bessel functions.

The positivity arises often as a critical issue in various disguises. For example, it is known that
the radial function $\,F({\bf x}) = f(\|{\bf x}\|), \,{\bf x}\in\R^d,$ is positive definite on $\R^d$ if
its Fourier transform is nonnegative. In view of the well-known Fourier transform formula
 \begin{align*}
 \widehat{F}(\xi) = (2\pi)^{\frac d2 }\|\xi\|^{1-\frac d2} \int_0^\infty t^{\frac d2} f(t)  J_{\frac{d-2}{2}}(\|\xi\| t)\,dt \qquad (\xi\in\R^d)
 \end{align*}
 and the fact that any Hankel transform of order $\,\nu>\mu>-1\,$ can be written as the Hankel transform of order $\mu$,
it is  positivity of Hankel transforms of order $\,\nu\ge  (d-2)/2\, $ that plays a decisive role
 in determining the positive definiteness of radial functions on $\R^d$ (see \cite{SW, We} for details).

The present paper is organized as follows. We revisit Sturm's oscillation theory in section 2 with our focuses on the distribution
of zeros and the wave of arches of oscillatory solutions. We establish P\'olya-type positivity criteria for the oscillatory integral
\eqref{OT1} in section 3, one of which will be illustrated with Fourier transforms in section 4. We present main applications to Hankel
transforms in section 5 and discuss $Y$-transforms in the last section.

\section{Sturm's oscillation theory}
This section aims to revisit some of Sturm's oscillation theory, original and extended,
which are relevant to the present work. As standard, given an interval $I$,
a function is said to be \emph{oscillatory} on $I$
when it changes sign infinitely often. If the sign change occurs at most finitely many times,
it is said to be \emph{non-oscillatory} on $I$. In what follows, it will be assumed that the underlying
function $\phi(t)$ in \eqref{ST1} is continuous for $\,t>0.$

\subsection{Oscillations with scattered zeros}
One of classical Sturm's theorems states that if there exist $\,a>0\,$ and $\,m>0\,$ such that
$\,\phi(t)\ge m\,$ for all $\,t\ge a,$ then any non-trivial solution of \eqref{ST1} is oscillatory on the interval $[a, \infty)$.
As an alternative, it is Kneser \cite{K}, Fite \cite{F} and Hille \cite{H} who proposed the following weaker conditions:
\begin{align*}
&{\rm(A1)}\quad  \liminf_{t\,\to\,\infty}\big[t^2\phi(t)\big]>\frac 14.\\
&{\rm(A2)} \quad \int_a^\infty \phi(t) dt = +\infty.\\
&{\rm(A3)}\quad  \liminf_{t\,\to\,\infty}\left[t\int_t^\infty\phi(s) ds\right]>\frac 14.
\end{align*}

\begin{lemma}\label{LemmaST1}
Suppose that $\,\phi(t)\ge 0\,$ for all $\,t\ge a,$ where $\,a>0,$ and one of assumptions {\rm (A1)}, {\rm (A2)}, {\rm (A3)} holds true.
Then any non-trivial solution $u(t)$ of \eqref{ST1} is oscillatory on $[a, \infty).$
Moreover, $u(t)$ has an infinity of zeros which are all simple, not concentrated on a finite subinterval of $[a, \infty)$
but scattered over $[a, \infty)$. In particular, the zeros of $u(t)$ are enumerable in ascending
order of magnitude on $[a, \infty).$
\end{lemma}

The proof for oscillation is simple and available in many textbooks or expository articles
(see \cite{Si, Wk}, for instance). The simplicity of each zero is a consequence of the existence and uniqueness theorem of Picard.
If the zeros of $u(t)$ are concentrated on a finite subinterval, then it follows by the Bolzano-Weierstrass theorem
that the zeros have an accumulation point at which both $u(t), u'(t)$ vanish. By Picard's theorem, $\,u(t) \equiv 0,$ which leads to a contradiction.
It is thus proved that the zeros are not concentrated on a finite interval but scattered over the interval $[a, \infty)$.

\begin{remark}
If $u(t)$ is oscillatory on $(0, b],\,b>0,$ then the zeros of $u(t)$ has an accumulation point $0$
and hence may not be enumerable in ascending order of magnitude.
In order to avoid such an occasion, we recall that Kneser and Hille also proved that if $\,\phi(t)\ge 0\,$ for all $\,t\ge a>0\,$ and
$$ \limsup_{t\,\to\,\infty}\big[t^2\phi(t)\big]<\frac 14\quad\text{or}\quad
\limsup_{t\,\to\,\infty}\left[t\int_t^\infty\phi(s) ds\right]<\frac 14,$$
then any non-trivial solution of \eqref{ST1} is non-oscillatory on $[a, \infty).$

By considering $\,v(t) = t u(1/t)\,$ which obviously satisfies
\begin{equation*}
v'' + \frac{1}{t^4}\phi\left(\frac 1t\right) v =0,\quad t>0,
\end{equation*}
it is simple to infer from the above non-oscillation criteria that
any non-trivial solution of \eqref{ST1} is non-oscillatory on the interval $\,(0, b],\,b>0,$
when $\phi(t)$ satisfies either of the following.
\begin{align*}
&{\rm(B1)} \quad  \limsup_{t\,\to\, 0+} \big[t^2\phi(t)\big] <\frac 14.\\
&{\rm(B2)} \quad \limsup_{t\,\to\, 0+} \left[\frac 1t \int_0^t s^2 \phi(s)\,ds\right]<\frac 14.
\end{align*}
\end{remark}

\subsection{Oscillation magnitudes}
As for the distribution of zeros and \emph{relative magnitudes of oscillation}, Sturm's comparison
theorems state as follows (see  \cite{LM, Ma1, S, Wa}).

\medskip

Let $\big(v(t), w(t)\big)$ be a $C^1$ solution of the system of equations
\begin{equation*}
\left\{\begin{aligned}
&{\,v'' + \phi(t) v =0,}\\
&{\,w'' +\psi(t) w =0,}\end{aligned}\right. \quad t\in I,
\end{equation*}
where $I$ is an open interval and $\phi(t), \psi(t)$ are continuous for $\,t\in I.$

\begin{itemize}
\item[\rm(I)] If $\,\phi(t)>\psi(t)\,$ for $\,t\in I,$ then $v(t)$ has at least one zero between two consecutive zeros of $w(t)$.
As a consequence, if $\,0<m<\phi(t)<M\,$ for $\,t\in I\,$ and $t_1, t_2$ are consecutive zeros of $v(t)$ with
$\,t_1<t_2,$ then
$$\frac{\pi}{\,\sqrt M\,}<t_2-t_1<\frac{\pi}{\,\sqrt m\,}.$$

\item[\rm(II)] For $\, a\in I,$ suppose that $\,v(a) = w(a) =0\,$ and $\,v'(a+) = w'(a+) >0.$
If $\,\phi(t)<\psi(t)\,$ for $\,t>a,\,t\in I,$ then $\,v(t)>w(t)\,$ for all $t$ between
$a$ and the first zero of $w(t)$ and hence the first zero of $w(t)$ for $\,t>a\,$ is on the
left of the first zero of $v(t)$ for $\,t>a.$
\end{itemize}

\begin{lemma}\label{lemmaA}
Suppose that $u(t)$ is a nontrivial solution of \eqref{ST1} satisfying $\,u(0+)=0$,
where $\phi(t)$ is continuous, strictly increasing and subject to
one of assumptions {\rm (A1), (A2), (A3).} Then the following holds true.

\begin{itemize}
\item[\rm(i)] $u(t)$ is oscillatory on $(0, \infty)$ with an infinity of simple zeros
which are enumerable in ascending order of magnitude.

\item[\rm(ii)] The wave of arches of graph $\,s=u(t)\,$ decreases in the sense
\begin{equation*}
 \left\{\begin{aligned}
 &{\quad \zeta_{k}-\zeta_{k-1}>\zeta_{k+1}-\zeta_k,}\\
&{\,\big|u(\zeta_k-t)\big|>\big|u(\zeta_k+t)\big|,\quad 0<t<\zeta_{k+1}-\zeta_k,}
\end{aligned}\right.
\end{equation*}
where $\big(\zeta_k\big)$ denotes all zeros of $u(t),\,
0\equiv \zeta_0<\zeta_1<\zeta_2<\cdots.$
\end{itemize}
\end{lemma}

\begin{proof}
As $\phi(t)$ is strictly increasing, either $\phi(t)$ has a unique zero $\,a>0\,$
or $\,\phi(t)\ge 0\,$ for all $\,t>0.$ In the former case, since it is known that $u(t)$ has at most one zero
for $\,0<t<a\,$ due to the sign condition $\,\phi(t)<0\,$ (see, for instance, \cite{Si}), the stated
result of part (i) follows from Lemma \ref{LemmaST1}.

In the latter case, we note that $\,\phi(0+)\,$ exists and hence
$$\limsup_{t\,\to\, 0+} \big[t^2\phi(t)\big] =0<\frac 14,$$
that is, condition (B1) is verified. As a consequence, for any fixed $\,b>0,$
$u(t)$ is non-oscillatory on the interval $\,(0, b]\,$ and oscillatory on the interval $\,[b, \infty)\,$ so that part (i) follows once again by Lemma \ref{LemmaST1}.

By considering $-u(t)$ otherwise, we may assume $\,u(t)>0$ for $0<t<\zeta_1$ and hence $u(t)$ is positive on $\,\big(\zeta_{2k},\,\zeta_{2k+1}\big)\,$
and negative on $\,\big(\zeta_{2k+1},\,\zeta_{2k+2}\big)\,$ for each $\,k=0,1,\cdots.$ We fix $\,k\ge 1\,$ and consider auxiliary functions
\begin{align*}
v(t) &= (-1)^{k+1} u\left(\zeta_{k} -t\right),\quad 0\le t< \zeta_{k}-\zeta_{k-1},\\
w(t) &= (-1)^k u\left(\zeta_{k} +t\right),\quad 0\le t< \zeta_{k+1}-\zeta_{k},
\end{align*}
both of which are positive. Evidently, $v(t)$ and $w(t)$ satisfy
\begin{align*}
&\quad  v'' + \phi\left(\zeta_{k} -t\right)v =0, \quad w'' + \phi\left(\zeta_{k} +t\right)w =0,\\
& v(0) = w(0) =0,\quad v'(0) = w'(0)  = (-1)^ku'(\zeta_{k})>0,
\end{align*}
whence part (ii) follows at once from Sturm's theorem (II).
\end{proof}

\begin{remark}\label{extraRK1}
When $\phi(t)$ is increasing, not necessarily in the strict sense,
an obvious modification indicates that part (ii) must be replaced by
\begin{equation}\label{R1}
 \left\{\begin{aligned}
 &{\quad \zeta_{k}-\zeta_{k-1}\ge \zeta_{k+1}-\zeta_k,}\\
&{\,\big|u(\zeta_k-t)\big|\ge \big|u(\zeta_k+t)\big|,\,\, 0<t<\zeta_{k+1}-\zeta_k.}
\end{aligned}\right.
\end{equation}

The inequalities of part (ii) as well as \eqref{R1} are reversed in the case when $\phi(t)$
is decreasing. Although these results, often referred to as Sturm's convexity theorem,  are known
and available in the aforementioned references, we gave a proof for part (ii) for the sake of completeness.
\end{remark}

\section{Oscillatory integrals}
Our main theorem reads as follows.

\begin{theorem}\label{theoremM1}
Suppose that $u(t)$ is an oscillatory solution of \eqref{ST1},
where $\phi(t)$ is continuous, strictly increasing
and subject to one of assumptions {\rm (A1), (A2), (A3).} Let $\big(\zeta_k\big)$ denote
all zeros of $u(t)$ arranged in ascending order and assume that
$\,u(0+)=0$ and $\,u(t)>0\,$ for $\,0<t<\zeta_1.$ Put
\begin{equation*}
(Uf)(x)= \int_0^\infty f(t) u(xt)\,dt\qquad (x>0).
\end{equation*}
 If $f(t)$ is a nonnegative, sectionally continuous and decreasing function such that it does not vanish
identically and satisfies
\begin{align*}
&{\rm (P1)}\quad \int_0^1 f(t) u(t)\,dt<\infty,\\
&{\rm (P2)}\quad \lim_{k\to \infty}\int_{\zeta_k}^{\zeta_{k+1}} f(t/x) |u(t)|\, dt =0\quad\text{for each $\,x>0,$}
\end{align*}
then $(Uf)(x)$ exists for each $\,x>0\,$ with $\,(Uf)(x)>0.$
\end{theorem}

\begin{proof} The sign assumption on $u(t)$ for $\,0<t<\zeta_1\,$ implies
\begin{equation}\label{I1}
\left\{\begin{aligned}
&{u(t)>0,\quad\zeta_{2k}<t<\zeta_{2k+1},}\\
&{u(t)<0,\quad\zeta_{2k+1}<t<\zeta_{2k+2},}
\end{aligned}\right.\,\,\,\, k=0,1,2,\cdots,
\end{equation}
where $\,\zeta_0 \equiv 0.$ By Lemma \ref{lemmaA}, $\,\zeta_{2k+1}-\zeta_{2k}>\zeta_{2k+2}-\zeta_{2k+1}\,$ and
\begin{align}\label{I2}
u(\zeta_{2k+1}-t)>\big|u(\zeta_{2k+1}+t)\big|,\quad 0<t<\zeta_{2k+2}-\zeta_{2k+1}.
\end{align}

Let us fix $\,x>0\,$ and put $\,g(t) = f(t/x),\,t>0.$ We decompose
\begin{align*}
x(Uf)(x) &= \sum_{k=0}^\infty \int_{\zeta_k}^{\zeta_{k+1}} g(t) u(t) dt
=\sum_{k=0}^\infty B_{k},\\
B_{k} &:= \left(\int_{\zeta_{2k}}^{\zeta_{2k+1}} +
\int_{\zeta_{2k+1}}^{\zeta_{2k+2}}\right) g(t) u(t) dt,
\end{align*}
which is legitimate if $(Uf)(x)$ exists. By changing variables, we write
\begin{align}\label{I4}
B_{k} &=\int_0^{\zeta_{2k+1}-\zeta_{2k}} g(\zeta_{2k+1} -t) u(\zeta_{2k+1} -t)\,dt\nonumber\\
 & + \int_0^{\zeta_{2k+2}-\zeta_{2k+1}} g(\zeta_{2k+1} +t) u(\zeta_{2k+1} +t)\,dt\nonumber \\
 &=\int_0^{\zeta_{2k+2}-\zeta_{2k+1}}\Big[g(\zeta_{2k+1} -t) u(\zeta_{2k+1} -t)
 + g(\zeta_{2k+1} +t) u(\zeta_{2k+1} +t)\Big]dt\nonumber\\
 &+ \int_{\zeta_{2k+2}-\zeta_{2k+1}}^{\zeta_{2k+1}-\zeta_{2k}} g(\zeta_{2k+1} -t) u(\zeta_{2k+1} -t)\,dt.
\end{align}

Since $g(t)$ is nonnegative, decreasing and sectionally continuous,
\eqref{I1}, \eqref{I2} clearly indicate that both integrands of \eqref{I4} are
positive when the interval $\big[\zeta_{2k},\,\zeta_{2k+2}\big]$ contains the support of $g(t)$
and nonnegative in the general case. Thus it is evident that $\,B_k>0\,$ in the former case and
$\,B_k\ge 0\,$ for all $k$. By modifying the support condition in an obvious way,
we may conclude that $\,(Uf)(x)>0\,$ unless $g(t)$ vanishes identically.

It remains to verify that $(Uf)(x)$ exists, that is, the integral defining $(Uf)(x)$ converges.
For this purpose, we decompose
\begin{equation}\label{I5}
x(Uf)(x) = \sum_{k=0}^\infty (-1)^k A_k,\quad A_k := \int_{\zeta_k}^{\zeta_{k+1}} g(t) |u(t)| dt.
\end{equation}
Since $f(t)$ is monotone and $u(t)$ is continuous for $\,t>0,$ the condition (P1) implies that
the integral of $f(t)|u(xt)|$, possibly improper, over any right-neighborhood of $0$ exists. It is thus evident that
$$A_0 = x\int_0^{\zeta_1/x} f(t)|u(xt)| dt <\infty.$$

We may assume that $\,A_k>0\,$ for all $k$ for the convergence is obvious otherwise.
Proceeding as above, we observe that
\begin{align*}
A_k - A_{k+1} = & (-1)^k \int_0^{\zeta_{k+2}-\zeta_{k+1}}\Big[g(\zeta_{k+1} -t) u(\zeta_{k+1} -t)\\
 &\qquad\qquad\qquad\quad  + \,g(\zeta_{k+1} +t) u(\zeta_{k+1} +t)\Big]dt\\
 &+ (-1)^k \int^{\zeta_{k+1}-\zeta_{k}}_{\zeta_{k+2}-\zeta_{k+1}} g(\zeta_{k+1} -t) u(\zeta_{k+1} -t)\,dt>0,
\end{align*}
where $\, k=0, 1, 2, \cdots,$ and so $(A_k)$ is strictly decreasing. Since $\,A_k\to 0\,$ as $\,k\to\infty\,$ by condition (P2),
it follows by the alternating series test that the series on the right of \eqref{I5} converges and thus $(Uf)(x)$ exists.
\end{proof}

\begin{remark}
If we assume instead that
$\,u(t)<0\,$ for all $\,0<t<\zeta_1,$ without altering other assumptions,
then the same proof verifies
$$\,(Uf)(x)<0\quad\text{for all}\quad x>0.$$
\end{remark}

Suppose that $\phi(t)$ is merely increasing so that \eqref{R1} holds true. In this case, a close inspection on
\eqref{I4} shows that both integrands are nonnegative so that $\,B_k\ge 0\,$ for all $k$ under the same assumption on $g(t)$.
If the interval $\big[\zeta_{2k},\,\zeta_{2k+2}\big]$ contains the support of $g(t)$, where $g(t)$ is assumed to be strictly decreasing,
then the first integrand is positive and hence $\,B_k>0.$

\begin{theorem}\label{theoremM3}
Let $u(t)$ be an oscillatory solution of \eqref{ST1},
where $\phi(t)$ is continuous, increasing, not necessarily in the strict sense, and subject to
one of assumptions {\rm (A1), (A2), (A3),} such that
$\,u(0+)=0$ and $\,u(t)>0\,$ for all $\,0<t<\zeta_1.$
If $f(t)$ is a nonnegative and sectionally continuous function such that
it is strictly decreasing on its support and subject to conditions {\rm (P1), (P2)}, then $(Uf)(x)$ exists for each $\,x>0\,$ with $\,(Uf)(x)>0.$
\end{theorem}

\section{Fourier transforms}
This section aims to refine P\'olya's results concerning positivity of
Fourier transforms by revisiting in the framework of oscillatory integrals.

As for the Fourier sine transform on the positive real line
\begin{equation*}
(Uf)(x) =\int_0^\infty f(t) \sin\, (xt)\,dt \qquad(x>0),
\end{equation*}
the kernel $\,u(t) = \sin t\,$ is an oscillatory solution of
the differential equation \eqref{ST1} with $\,\phi(t) \equiv 1\,$ such that
$\,u(0+) =0\,$ and $\,u(t)>0\,$ for $\,0<t<\pi.$ Thus the Fourier sine transform
falls under the scope of Theorem \ref{theoremM3}.

\begin{itemize}
\item Due to the asymtotic behavior $\,\sin t \,\sim\,t\,$
as $\,t\to 0,$ it is evident that condition (P1) holds when $\,f(t) t\,$ is integrable over the interval $(0, 1).$

\item We claim that condition (P2) holds if $f(t)$ is a nonnegative decreasing function such that $\,f(t)\to 0\,$ as $\,t\to \infty.$ Indeed,
given $\,x>0\,$ and $\,\epsilon>0,\,$ if we choose $N$ so that $\,f(t/x)<\epsilon/2\,$ for all $\,t\ge N\pi,$ then
$$\int_{k\pi}^{(k+1)\pi} f(t/x) |\sin t|\,dt\le \frac{\epsilon}{2}\int_{k\pi}^{(k+1)\pi}  |\sin t|\,dt =\epsilon,\quad k\ge N.$$
\end{itemize}

To sum up, Theorem \ref{theoremM3} yields the following result.

\begin{theorem}\label{theoremT}
The inequality
\begin{equation*}
\int_0^\infty f(t) \sin (xt)\,dt >0
\end{equation*}
holds true for all $\,x>0\,$ when $f(t)$ is nonnegative, sectionally continuous, strictly decreasing on its support,
\begin{equation*}
\int_0^1 f(t) t\,dt<\infty \quad \text{and}\quad \lim_{t\to\infty} f(t) =0.
\end{equation*}
\end{theorem}

We illustrate Theorem \ref{theoremT} with
\begin{equation*}
\int_0^\infty \frac{\sin(xt)}
{\,t^{\gamma} (t^2 + a^2)^{\delta}}\,dt >0,\quad x>0,\,a>0,
\end{equation*}
valid for $\,0\le\gamma <2,\, \delta>0\,$ or $\,0<\gamma<2,\,\delta =0,$ which includes the following
special cases of interest (\cite[(20), (28), 2.2]{E})
\begin{align*}
\int_0^\infty \frac{\sin(xt)}
{t(t^2 + a^2)}\,dt &= \frac{\pi}{2 a^2} \left(1- e^{-ax}\right),\\
\int_0^\infty \frac{\sin(2xt)}
{\sqrt{t(t^2 + a^2)}}\,dt &= \sqrt{\pi x}\, I_{1/4}(ax) K_{1/4}(ax),
\end{align*}
where $\,I_\nu(x), K_\nu(x)\,$ stand for the modified Bessel functions of the first and second kind, respectively
(see, for instance, \cite[3.7]{Wa}).

An integration by parts gives
$$\int_0^\infty f(t)\cos (xt)\,dt = \frac 1x \int_0^\infty \left[-f'(t)\right] \sin (xt)\,dt,$$
provided that $\,t f(t)\to 0\,$ as $\,t\to 0+\,$ and $\,f(t)\to 0\,$ as $\,t\to \infty.$
By applying Theorem \ref{theoremM3} to $\,-f'(t),$  we obtain the following result. As customary,
a function $h$ will be said to be sectionally smooth on an interval $I$ if both $\,h, h' \,$ are sectionally or piecewise continuous on $I$.

\begin{corollary}\label{corollaryT}
The inequality
\begin{equation*}
\int_0^\infty f(t)\cos (xt)\,dt>0
\end{equation*}
holds true for all $\,x>0\,$ when $f(t)$ is nonnegative, sectionally smooth,  strictly decreasing and convex on its support,
$$\int_0^1 f(t)\,dt<\infty\quad\text{and}\quad \lim_{t\to\infty} f(t) =0.$$
\end{corollary}

We illustrate Corollary \ref{corollaryT} with
\begin{equation*}
\int_0^\infty \frac{\cos (xt)}{(t+a)^\lambda}\,dt>0,\quad x>0,\,a>0,
\end{equation*}
valid for any $\,\lambda>0.$ As this example may indicate, it should be noted that
Theorem \ref{theoremT} and Corollary \ref{corollaryT} extend P\'olya's results
to the case of non-integrable functions. Additionally relevant examples are
\begin{align*}
\int_0^\infty t^{-\nu}\sin (xt)\,dt &=  \frac{\pi}{2\Gamma(\nu)}
\csc \left(\frac{\nu\pi}{2}\right) x^{\nu-1}>0,\\
\int_0^\infty t^{-\nu}\cos (xt)\,dt &= \frac{\pi}{2\Gamma(\nu)}
\sec \left(\frac{\nu\pi}{2}\right) x^{\nu-1}>0,
\end{align*}
valid for all $\,x>0,$ where the first inequality holds true for $\,0<\nu<2\,$ and the second inequality
for $\,0<\nu<1\,$ (see \cite[(1), 1.3, (1), 2.3]{E} ). As obvious, the function $\,f(t) = t^{-\nu}\,$
is not integrable for any $\,\nu\in\R.$

\section{Hankel transforms}
In the framework of Sturm's oscillation theory, the kernel $\,u(t) =\sqrt t J_\nu(t)\,$ of Hankel transform
\eqref{H1} arises as a solution of \eqref{ST1} with
\begin{equation}\label{SB1}
\phi(t) = 1 + \frac{1-4\nu^2}{4t^2},\quad t>0.
\end{equation}

As a smooth function, $\phi(t)$ satisfies all oscillation assumptions (A1), (A2), (A3). Moreover,
it decreases strictly to the value $1$ when $\,|\nu|<1/2\,$ and increases strictly to the value $1$ when
$\,|\nu|>1/2.$ As customary, $\big(j_{\nu, k}\big)$ will denote all positive zeros of $J_\nu(z)$
arranged in ascending order.

Our principal theorem states as follows.

\begin{theorem}\label{theoremH1}
For $\,\nu>1/2,$ we have $\,(H_\nu f)(x)>0\,$ for all $\,x>0\,$
if $f(t)$ is a nonnegative, sectionally continuous and decreasing function for $\,t>0\,$ such that
it does not vanish identically,
$$\int_0^1 t^{\nu+1/2}f(t)dt<\infty \quad \textit{and}\quad \lim_{t\to \infty}f(t)=0.$$
\end{theorem}

\begin{proof}
On considering the asymptotic behavior
\begin{align*}
u(t)=\sqrt t J_\nu(t) &= \sqrt t\, \sum_{m=0}^\infty \frac{(-1)^m}{\,m!\,\Gamma(\nu+1+m)\,}\left(\frac t2\right)^{2m+\nu}\\
&\,\sim\,\frac{\sqrt 2}{\Gamma(\nu+1)} \left(\frac t2\right)^{\nu+1/2}\quad\text{as $\,t\to 0,$}
\end{align*}
it is clear that $\,u(0+) =0, \,u(t)>0\,$ for $\,0<t<j_{\nu, 1}\,$ and condition (P1) holds true.
By Theorem \ref{theoremM1}, hence, it suffices to show that condition (P2) holds true under the
stated assumptions on $f(t)$.

Since $\phi(t)$ increases strictly to the value $1$, if we apply Sturm's theorem (I) with any $m$ satisfying
$\,0<m<\phi(t)<1\,$ and then letting $\, m\to 1,$ it is straightforward to deduce that
\begin{equation}\label{PH1}
j_{\nu, 2} - j_{\nu, 1}>j_{\nu, 3}-j_{\nu, 2} >\cdots\,\to\,\pi.
\end{equation}

Hankel's asymptotic expansion (\cite[7.21]{Wa}) indicates
\begin{equation}\label{PH2}
u(t) = \sqrt{\frac{2}{\pi}}\,\cos \left(t- \frac{\nu\pi}{2}-\frac{\pi}{4}\right) +  O\big(t^{-1}\big)
\end{equation}
as $\,t\to \infty.$ By combining \eqref{PH1} with McMahon's asymptotic expansion for large zeros
of Bessel functions (\cite {Mc}, \cite[15.53]{Wa})
\begin{equation}\label{ZA}
j_{\nu, k} = \left(k+ \frac \nu 2 - \frac 14\right)\pi + O\big(k^{-1}\big)
\end{equation}
as $\,k\to \infty,$ we may choose an integer $N$ so that
$$j_{\nu, k}\ge \frac{k\pi}{2},\quad \pi<j_{\nu, k+1}-j_{\nu, k}<2\pi\quad\text{for all $\,k\ge N.$}$$

For $\,k\ge N,$ these inequalities clearly imply that
\begin{align*}
\int_{j_{\nu, k}}^{j_{\nu, k+1}} \left| \cos \left(t- \frac{\nu\pi}{2}-\frac{\pi}{4}\right)\right|\,dt \le 2\pi,\quad
\int_{j_{\nu, k}}^{j_{\nu, k+1}} \frac{\,dt\,}{t}  \le \frac{4}{k},
\end{align*}
whence it follows from Hankel's formula \eqref{PH2} that
$$\lim_{k\to\infty} \int_{j_{\nu, k}}^{j_{\nu, k+1}} f(t/x) |u(t)|\,dt =0, \quad x>0,$$
under the assumption $\,f(t)\to 0\,$ as $\,t\to\infty.$ It is thus verified that (P2) holds true and
Theorem \ref{theoremM1} gives the desired positivity.
\end{proof}

As readily observed, Theorem \ref{theoremH1} gives
\begin{equation*}
\int_0^\infty \frac{J_\nu(xt) \sqrt{xt}\,}{\,t^\gamma(t^2 + a^2)^\delta} \,dt\,>0,\quad x>0,\,a>0,
\end{equation*}
valid for $\,\nu\ge 1/2\,$ and $\,\gamma =0,\,\delta>0\,$ or $\,0<\gamma<\nu+3/2,\,\delta\ge 0,$ which
includes the following special cases  \cite[(3), (11), (17), 8.5]{E} :
\begin{align*}
\int_0^\infty J_\nu(xt)\,dt &= \frac{1}{x},\\
\int_0^\infty \frac{J_\nu(2xt)\,dt\,}{\sqrt{t^2 + a^2}}\,&=
I_{\nu/2}(ax) K_{\nu/2}(ax),\\
\int_0^\infty \frac{J_\nu(2xt)\,dt\,}{\, t^\nu (t^2 + a^2)^{\nu+1/2}}\, &= \frac{(2/a)^{2\nu}\Gamma(\nu+1)}
{\Gamma(2\nu+1)}\,x^{\nu}  I_{\nu}(ax) K_{\nu}(ax).
\end{align*}

By applying instead Theorem \ref{theoremF}, to be established in the sequel,
it is possible to extend the stated range of positivity to the case $\,\nu>-1.$

\subsection{Alternative: The case $\,0<\nu<1/2$}
While we have investigated positivity of Hankel transforms in the cases
$\,\nu>1/2\,$ or $\,\nu=\pm 1/2\,$ on the basis of comparing the wave of arches,
such a method is no longer suitable in dealing with other cases due to our restrictive assumptions
on the kernel and $\phi(t)$.  As an alternative, we consider the kernel of type $\,u(t) = \sqrt t J_\nu\left(t^\alpha\right)\,$
which arises as a solution of \eqref{ST1} with
\begin{align}\label{AH1}
\phi(t) = \frac{1}{t^2}\left(\alpha^2 t^{2\alpha} + \frac{1-4 \alpha^2\nu^2}{\,4\,}\right), \quad t>0.
\end{align}

When $\,\alpha>1,\,\alpha|\nu|\ge 1/2,$ $\phi(t)$ is strictly increasing. Since
\begin{equation*}
u(t) = \sqrt t J_\nu\left(t^\alpha\right) \,\sim\,\frac{ t^{\alpha\nu+1/2}\,}{2^\nu \Gamma(\nu+1)} \quad\text{as $\,t\to 0+,\,$}
\end{equation*}
if $\,\alpha\nu+1/2>0,$ then $\,u(0+) =0, \,u(t)>0\,$ for $\,0<t<j_{\nu, 1}^{1/\alpha}\,$ and
(P1) is equivalent to the integrability of $\,f(t)  t^{\alpha\nu+1/2}\,$ for $\,0<t<1.$

In view of Hankel's asymptotic formula
\begin{equation*}
u(t) = \sqrt{\frac{2}{\pi}}\,t^{\frac{1-\alpha}{2}}\left[\cos \left(t^\alpha- \frac{\nu\pi}{2}-\frac{\pi}{4}\right) +  O\big(t^{-\alpha}\big)\right]
\end{equation*}
as $\,t\to \infty\,$ and a modification of \eqref{PH1} in the case $\,|\nu|\le 1/2,$
\begin{equation*}
j_{\nu, 2} - j_{\nu, 1}\le j_{\nu, 3}-j_{\nu, 2} \le \cdots\,\to\,\pi,
\end{equation*}
it is easy to confirm that (P2) holds with no further assumptions on $f(t)$.

It is thus shown by Theorem \ref{theoremM1} that
\begin{equation*}
(Uf)(x) = \int_0^\infty f(t) J_\nu\big[(xt)^\alpha\big] \sqrt{xt}\,dt>0 \qquad (x>0),
\end{equation*}
valid for $\,\nu>0,\,\alpha>1\,$ with $\,\alpha\nu\ge 1/2,$
provided that $f(t)$ is a nonnegative, sectionally continuous and decreasing function satisfying
$$ \int_0^1 f(t)\, t^{\alpha\nu + 1/2}\,dt <\infty.$$

By changing variables appropriately, we find that
\begin{equation*}
\alpha x^{\frac{\alpha-1}{2\alpha}} (Uf)(x^{\frac 1\alpha})
= \int_0^\infty f(t^{\frac 1\alpha}) t^{\frac{3(1-\alpha)}{2\alpha}} J_\nu(xt) \sqrt{xt}\,dt>0.
\end{equation*}
In the case $\,0<\nu<1/2,$ if we choose $\,\alpha = 1/2\nu,$ then this alternative method yields
an extension of Theorem \ref{theoremH1} which reads as follows.

\begin{theorem}\label{theoremH2}
For $\,0<\nu< 1/2,$ we have $\,(H_\nu f)(x)>0\,$ for all $\,x>0\,$
if $f(t)$ is a nonnegative sectionally continuous function such that
it does not vanish identically, $\, t^{3/2-3\nu} f(t)\,$ is decreasing for $\,t>0\,$ and
$$  \int_0^1 t^{\nu+1/2}f(t)\,dt<\infty .$$
\end{theorem}

\begin{remark}
In \cite{Ma2}, Makai exploited the scaled function $\,u(t) =\sqrt t J_\nu\left(\beta t^{1/2\nu} \right)\,$
in proving that $\,j_{\nu, k}/\nu\,$ is a strictly decreasing function of $\,\nu>0.$
\end{remark}

We illustrate with Gegenbauer's formula
\begin{equation}\label{G}
\int_0^\infty e^{-bt} J_\nu(xt) \frac{dt}{t} = \frac{\big[\sqrt{x^2 + b^2} - b\big]^\nu\,}{\nu x^\nu}>0,
\end{equation}
valid for all $\,x>0\,$ and $\,b>0,\,\nu>0\,$  (see \cite[(2), 8.6]{E}, \cite[13.2]{Wa}).

\begin{itemize}
\item In the case $\,\nu\ge 1/2,$ we identify
\begin{align*}
\int_0^\infty e^{-bt} J_\nu(xt) \frac{dt}{t} = \frac{1}{\sqrt x}\big(H_\nu f\big)(x),\,\,
f(t) = t^{-\frac 32} e^{-bt}.
\end{align*}
As elementary, $f(t)$ is positive, continuous, strictly decreasing and $\,f(t)\to 0\,$ as $\,t\to\infty.$ In addition,
it is trivial to see that
$$\int_0^1 t^{\nu+1/2} f(t) dt =\int_0^1 t^{\nu-1} e^{-bt} dt<\infty.$$
Thus it follows by theorems \ref{theoremT}, \ref{theoremH1} that $\,\big(H_\nu f\big)(x)>0.$

\item In the case $\,0<\nu<1/2,$ we note that the associated function
$$t^{3/2-3\nu} f(t) = t^{-3\nu} e^{-b t},\quad t>0, $$
is decreasing and hence $\,\big(H_\nu f\big)(x)>0\,$ by Theorem \ref{theoremH2}.
\end{itemize}

\subsection{The general case $\,\nu>-1$}
By exploiting the recurrence relation of Bessel functions, it is possible to deduce positivity of Hankel transforms
in the general case $\,\nu>-1\,$ on the basis of the preceding two theorems.

\begin{theorem} \label{theoremF}
For $\,\nu>-1,$ suppose that $f(t)$ is a nonnegative, decreasing and sectionally smooth function for $\,t>0\,$ satisfying
\begin{equation}\label{GP2}
\int_0^1 t^{\nu+1/2}f(t)dt<\infty \quad \textit{and}\quad \lim_{t\to \infty}f(t)=0.
\end{equation}
Put
\begin{align*}
g(t) &= -t^{\nu+1/2} \frac{d}{dt} \left[ t^{-\nu-1/2} f(t)\right]\\
&= \left(\nu+\frac12\right) \frac{f(t)}{t} - f'(t),\quad t>0.
\end{align*}
If both $f(t), g(t)$ do not vanish identically for $\,t>0,$ then $\,(H_\nu f)(x)>0\,$ for all $\,x>0\,$
under each of the following additional case-assumptions.

\begin{itemize}
\item[\rm(i)] For $\,\nu>-1/2,$ $g(t)$ is nonnegative, decreasing and $\, \lim_{t\to \infty} g(t)=0.$
\item[\rm(ii)] For $\,\nu = -1/2,$ $f(t)$ is strictly decreasing and convex on its support.
\item[\rm(iii)] For $\,\nu<-1/2,$ $g(t)$ is nonnegative and $t^{-3/2-3\nu} g(t)$ is decreasing.
\end{itemize}
\end{theorem}

\begin{proof}
By using the differentiation formula \cite[3.2]{Wa}
\begin{equation*}
    \frac{d}{dt} \left[t^{\nu+1}J_{\nu+1}(t)\right] = t^{\nu+1}J_{\nu}(t),\quad t>0,
\end{equation*}
we fix $\,x>0\,$ and integrate by parts to obtain
\begin{align}\label{GP3}
   \left(H_\nu f\right)(x)  & =\frac{1}{\sqrt x} \int_0^\infty t^{-\nu-1/2}f(t) \frac{d}{dt} \left[t^{\nu+1}J_{\nu+1}(xt)\right] dt\nonumber\\
    &=\frac{1}{\sqrt x} \,\Big[t^{1/2}f(t)J_{\nu+1}(xt)\Big]_{t=0}^\infty\nonumber\\
    &\quad\qquad -\frac 1x \int_0^\infty t^{\nu+1/2} \frac{d}{dt} \left[ t^{-\nu-1/2} f(t)\right] J_{\nu+1}(xt) \sqrt{xt}\,dt\nonumber\\
    &=\frac 1x \int_0^\infty g(t) J_{\nu+1}(xt) \sqrt{xt}\,dt =\frac{1}{x} \left(H_{\nu+1} g\right)(x),
\end{align}
provided that $\,G(t):= t^{1/2}f(t)J_{\nu+1}(xt)\,$ tends to zero as $t$ approaches zero or infinity.
To justify, we first note that $\,t^{1/2} J_{\nu+1}(xt)\,$ is bounded for large $t$, due to Hankel's formula \eqref{PH2},
and hence the second assumption of \eqref{GP2} indicates $\,G(t)\to 0\,$ as $\,t\to\infty.$ Next,
by inspecting the first assumption of \eqref{GP2} and monotonicity of $f(t)$,
it is not difficult to see that $\, t^{\nu+3/2} f(t) \to 0\,$ as $\,t\to 0+.$ By combining with the initial behavior
$$ t^{1/2} J_{\nu+1} (xt) \,\sim\, \frac{x^{\nu+1} t^{\nu+3/2}}{\Gamma(\nu+1) 2^{\nu+1}\,}\quad\text{as $\,t\to 0+,$}$$
we find that $\,G(t) \to 0\,$ as $\,t\to 0+\,$ and thus \eqref{GP3} is legitimate.

Concerning part (i), Theorem \ref{theoremH1} asserts that $\,(H_{\nu+1} g)(x)>0\,$ when $g(t)$ is nonnegative, sectionally continuous,
decreasing with
$$\int_0^1 t^{\nu+3/2} g(t) dt <\infty\quad\text{and}\quad \lim_{t\to\infty} g(t) =0.$$
On integrating by parts, it is easy to observe that
\begin{equation}\label{GP4}
\int_0^1 t^{\nu+3/2} g(t) dt = - f(1) + (2\nu+2)\int_0^1 t^{\nu+1/2} f(t) dt<\infty
\end{equation}
and the result follows. Part (ii) is equivalent to Corollary \ref{corollaryT}. In the case $\,\nu<-1/2,$ that is, $\,0<\nu+1<1/2,$
since Theorem \ref{theoremH2} indicates that $\,(H_{\nu+1} g)(x)>0\,$ when $g(t)$ is nonnegative, sectionally continuous, subject to condition \eqref{GP4}
and $\,t^{-3/2-3\nu} g(t)\,$ is decreasing, the result follows.
\end{proof}

As an illustration, we take $\,f(t) = t^{-\beta},\,t>0,$ and consider
\begin{align}\label{GP5}
(H_\nu f)(x) &= \int_0^\infty t^{-\beta} J_\nu(xt) \sqrt{xt}\,dt\nonumber\\
& =2^{-\beta+1/2}x^{\beta-1}\frac{\,\Gamma\big((2\nu-2\beta+3)/4\big) \,}{\,\Gamma\big((2\nu+2\beta+1)/4\big)\,},
\end{align}
valid for $\, \nu>-1,\,0<\beta<\nu+3/2\,$ (\cite[8.5, (7)]{E}, \cite[1.4, (4.6)]{Ob}). Since
$$g(t) = (\nu+\beta+1/2)t^{-\beta-1},\quad t>0,$$
it is elementary to find by applying Theorem \ref{theoremF} that the Hankel transform $(H_\nu f)(x)$ given by \eqref{GP5}
is positive for all $\,x>0\,$ if
\begin{equation*}
\left\{\begin{aligned}
&{\,\,\, 0<\beta<\nu +3/2\quad\text{for $\,\,\,\nu\ge -1/2,$}}\\
&{-\nu-1/2<\beta<\nu+3/2\quad\text{for $\,\, -1<\nu<-1/2,$}}
\end{aligned}\right.
\end{equation*}
the optimal range of $\,\nu, \beta\,$ for positivity in this example.

\section{Y-transforms}
As another oscillatory solution of the differential equation \eqref{ST1} with \eqref{SB1},
if we take  $\,u(t) = \sqrt t \,Y_\nu(t),$ where $Y_\nu(t)$ denotes a Neumann function or the Bessel function of the second kind
defined by
\begin{equation*}
    Y_\nu(t)=\frac{\,\cos{\nu\pi}J_\nu(t)-J_{-\nu}(t)\,}{\sin{\nu\pi}},
\end{equation*}
then the corresponding oscillatory integral takes the form
\begin{align}\label{NT1}
(N_\nu f)(x) &= \int_0^\infty f(t) Y_{\nu}(xt)\sqrt{xt}\,dt\nonumber\\
&=\cot \nu\pi\,(H_\nu f)(x) - \csc \nu\pi\,\left(H_{-\nu} f\right)(x),
\end{align}
often referred to as the $Y$-transform of order $\nu$ of $f(t)$ (see \cite{E}, \cite{Ob}).

Concerning positivity, we remark that Theorem \ref{theoremM1} is not applicable to
$Y$-transforms. Indeed, while $\phi(t)$ is increasing only when $\,|\nu|\ge 1/2,$ it is easy to see that
$\,|u(t)|= \sqrt t \,|Y_\nu(t)|\,$ tends to infinity as $t$ approaches zero when $\,|\nu|>1/2.$ Nevertheless,
it is possible to deduce positivity or negativity of $(N_\nu f)(x)$ by making use of
the relation \eqref{NT1} in the case $\,1/2<|\nu|<1,$ provided that both
$\, (H_\nu f)(x),\,(H_{-\nu} f)(x)\,$ are positive at the same time: It is positive for $\,-1<\nu<-1/2\,$
and negative for $\,1/2<\nu<1.$

In practice, it is straightforward to obtain the following positivity result for $Y$-transforms on
applying Theorem \ref{theoremM1} and Theorem \ref{theoremF}.

\begin{theorem}
For $\,1/2<|\nu|<1,$ let $f(t)$ be a nonnegative, decreasing and sectionally continuous function
such that it does not vanish identically,
$$\int_0^1 t^{-|\nu|+1/2}f(t) dt<\infty \quad \textit{and}\quad \lim_{t\to \infty}f(t)=0.$$
If $g(t)$ is nonnegative and $\,t^{-3/2 + 3|\nu|} g(t)\,$ is decreasing for $\,t>0,$ where
$$ g(t) = \left(-|\nu| + \frac 12\right) \frac{f(t)}{t} - f'(t),\quad t>0,$$
then the following inequalities hold true for each $\,x>0:$
\begin{equation*}
\left\{\begin{aligned}
&{\,\,\,(N_\nu f)(x)>0\quad\text{for $\,\,-1<\nu<-1/2,$}}\\
&{\,\,\,(N_\nu f)(x)<0\quad\text{for $\,\,1/2<\nu< 1.$}}
\end{aligned}\right.
\end{equation*}
\end{theorem}

\bigskip

\noindent
{\bf Acknowledgements.} We are grateful to anonymous referees for their suggestions which helped greatly
improve the original version of the present paper.

\bigskip

\noindent
{\bf Funding.} The research by Yong-Kum Cho was supported by
the National Research Foundation of Korea funded by
the Ministry of Science and ICT (2021R1A2C1007437). The research by Young Woong Park
was supported by the Chung-Ang University Graduate Research Scholarship in 2021.

\bigskip

\noindent
{\bf Disclosure Statement.} No potential conflict of interest was reported by the authors.

\end{document}